\documentclass[12pt]{amsart}
\usepackage{array,CJK}
\usepackage{curves}
\usepackage{amstext}
\usepackage{amsfonts}
\usepackage{amsmath}
\usepackage{amssymb}
\usepackage{longtable}

\newtheorem{theorem}{Theorem}[section]
\newtheorem{corollary}[theorem]{Corollary}

\newtheorem{proposition}[theorem]{Proposition}

\theoremstyle{definition}

\newtheorem{example}[theorem]{Example}
\theoremstyle{remark}
\newtheorem{remark}[theorem]{Remark}

\title[identities for the generalized Fibonacci polynomials] {Some identities for the generalized Fibonacci polynomials by the $Q(x)$ matrix}
\author[Chung-Chuan Chen]{Chung-Chuan Chen}
\author[Lin-Ling Huang]{Lin-Ling Huang}



\address{Department of Mathematics Education, National Taichung University of Education, Taiwan}
\email{chungchuan@mail.ntcu.edu.tw}
\email{linda266308@gmail.com}

\date{\today}
\begin{document}

\maketitle

\begin{abstract}
In this note, we obtain some identities for the generalized Fibonacci polynomial by using the $Q(x)$ matrix. These identities including the Cassini identity and Honsberger formula can be applied to some polynomial sequences, such as Fibonacci polynomials, Lucas polynomials, Pell polynomials,
Pell-Lucas polynomials, Fermat polynomials, Fermat-Lucas polynomials, and so on.
\end{abstract}

\baselineskip17pt

\section{Introduction}

A second order polynomial sequence $F_n(x)$ is said to be {\it the Fibonacci polynomial} if
for $n\geq2$ and $x\in\Bbb{R}$,
$$F_{n}(x)=xF_{n-1}(x)+F_{n-2}(x)$$
with $F_{0}(x)=0$ and $F_{1}(x)=1$. The Fibonacci polynomial and other polynomials
attracted a lot of attention over the last several decades (see, for instance, \cite{fp07,fp09,gould,h79,hm85,wz13}).
Recently, the generalized Fibonacci polynomial is introduced and studied intensely by many authors \cite{a94,a95,fhm18,fmm18},
which is a generalization of the Fibonacci polynomial.
Indeed, a polynomial sequence $G_n(x)$ in \cite{fhm18,fmm18} is called {\it the generalized Fibonacci polynomial} if for $n\geq2$,
$$G_{n}(x)=c(x)G_{n-1}(x)+d(x)G_{n-2}(x)$$
with initial conditions $G_{0}(x)$ and $G_{1}(x)$, where $c(x)$ and $d(x)$ are fixed non-zero polynomials in $\Bbb{Q}[x]$.
It should be noted that there is no unique generalization of Fibonacci polynomials.
Following the similar definitions in \cite{fmm18}, in this note, $\mathcal{F}_n(x)$ is said to be {\it the Fibonacci type polynomial} if for $n\geq2$,
$$\mathcal{F}_{0}(x)=0,\ \mathcal{F}_{1}(x)=a\ \mbox{and}\ \mathcal{F}_{n}(x)=c(x)\mathcal{F}_{n-1}(x)+d(x)\mathcal{F}_{n-2}(x)$$
where $a\in \Bbb{R}\setminus\{0\}$. If for $n\geq2$,
$$\mathcal{L}_{0}(x)=q,\ \mathcal{L}_{1}(x)=b(x)\ \mbox{and}\ \mathcal{L}_{n}(x)=c(x)\mathcal{L}_{n-1}(x)+d(x)\mathcal{L}_{n-2}(x),$$
then the polynomial sequence $\mathcal{L}_n(x)$ is called {\it the Lucas type polynomial}, where $q\in \Bbb{R}\setminus\{0\}$ and $b(x)$ is a fixed non-zero polynomial in $\Bbb{Q}[x]$. Naturally, both $\mathcal{F}_n(x)$ and $\mathcal{L}_n(x)$ are the generalized Fibonacci polynomials.
We note that if we assume $\mathcal{F}_{1}(x)=a=1$, then $\mathcal{F}_n(x)$ is the Fibonacci type polynomial given in \cite{fmm18}.
In addition, the definition of $\mathcal{L}_n(x)$ is the same with that of Fl\'orez et al \cite{fmm18} if $|q|=1$ or $2$, and $c(x)=\frac{2}{q}b(x)$.
In other words, our definitions of $\mathcal{F}_n(x)$ and $\mathcal{L}_n(x)$ are generalizations of those in \cite{fmm18}.

Since the investigation of identities for polynomial sequences $\mathcal{F}_n(x)$ and $\mathcal{L}_n(x)$ received
less attention than their numerical sequences, Fl\'orez et al \cite{fmm18} collected and proved many identities for both $\mathcal{F}_n(x)$ and $\mathcal{L}_n(x)$ by applying their Binet formulas mostly, when certain special initial conditions were satisfied for $\mathcal{F}_n(x)$ and $\mathcal{L}_n(x)$. These identities can be applied to Fibonacci polynomials, Lucas polynomials, Pell polynomials,
Pell-Lucas polynomials, Fermat polynomials, Fermat-Lucas polynomials, Chebyshev first kind polynomials, Chebyshev second kind polynomials, Jacobsthal polynomials, Jacobsthal-Lucas polynomials, and Morgan-Voyce polynomials. Indeed, all polynomial sequences in the upper part of Table \ref{table} below are
the Fibonacci type polynomials. On the other hand, those in the lower part of Table \ref{table} are the Lucas type polynomials. Table \ref{table} is the
rearrangement of \cite[Table 1]{fmm18}.

{\tiny
\begin{center}
\begin{table}[h]
\caption{}\label{table}
\begin{tabular}{|l|c|c|l|}
\hline
Polynomial & Initial value & Initial value & Recursive Formula\\
           & $G_0(x)$ & $G_1(x)$ & $G_n(x)=c(x)G_{n-1}(x)+d(x)G_{n-2}(x)$\\
\hline\hline
Fibonacci & $0$ & $1$ & $F_n(x)=xF_{n-1}(x)+F_{n-2}(x)$\\
Pell & $0$ & $1$ & $P_n(x)=2xP_{n-1}(x)+P_{n-2}(x)$\\
Fermat & $0$ & $1$ & $\Phi_n(x)=3x\Phi_{n-1}(x)-2\Phi_{n-2}(x)$\\
Chebyshev second kind & $0$ & $1$ & $U_n(x)=2xU_{n-1}(x)-U_{n-2}(x)$\\
Jacobsthal & $0$ & $1$ & $J_n(x)=J_{n-1}(x)+2xJ_{n-2}(x)$\\
Morgan-Voyce & $0$ & $1$ & $B_n(x)=(x+2)B_{n-1}(x)-B_{n-2}(x)$\\
Vieta & $0$ & $1$ & $V_n(x)=xV_{n-1}(x)-V_{n-2}(x)$\\
\hline
Lucas & $2$ & $x$ & $L_n(x)=xL_{n-1}(x)+L_{n-2}(x)$\\
Pell-Lucas & $2$ & $2x$ & $D_n(x)=2xD_{n-1}(x)+D_{n-2}(x)$\\
Pell-Lucas-prime & $1$ & $x$ & $D'_n(x)=2xD'_{n-1}(x)+D'_{n-2}(x)$\\
Fermat-Lucas & $2$ & $3x$ & $\vartheta_n(x)=3x\vartheta_{n-1}(x)-2\vartheta_{n-2}(x)$\\
Chebyshev first kind & $1$ & $x$ & $T_n(x)=2xT_{n-1}(x)-T_{n-2}(x)$\\
Jacobsthal-Lucas & $1$ & $1$ & $\Lambda_n(x)=\Lambda_{n-1}(x)+2x\Lambda_{n-2}(x)$\\
Morgan-Voyce & $2$ & $x+2$ & $C_n(x)=(x+2)C_{n-1}(x)-C_{n-2}(x)$\\
Vieta-Lucas & $2$ & $x$ & $v_n(x)=xv_{n-1}(x)-v_{n-2}(x)$\\
\hline
\end{tabular}
\end{table}
\end{center}}

In the note, by using the so called the $Q(x)$ matrix of Fibonacci type polynomials rather than the Binet formulas, we will obtain some new identities or recover some well-known ones including the Cassini identity and Honsberger formula for $\mathcal{F}_n(x)$ and $\mathcal{L}_n(x)$. In Section 2, we will present the results for the Fibonacci type polynomial $\mathcal{F}_n(x)$. Relying on Section 2, the identities of the Lucas type polynomial $\mathcal{L}_n(x)$ will be demonstrated in Section 3.

\section{Fibonacci type polynomials}

In this section, we will provide and prove some identities for the Fibonacci type polynomial $\mathcal{F}_n(x)$ by applying the Fibonacci type $Q(x)$ matrix. The original Fibonacci $Q$ matrix was introduced by Charles H. King in his master thesis (cf. \cite{koshy}), and given by
$$Q=
\left(\begin{matrix}
1 & 1 \\
1 & 0
\end{matrix}\right).$$
The Fibonacci $Q$ matrix is connected to the Fibonacci sequence $F_n$, which is defined as below
$$F_{0}=1,\ F_{1}=1\ \ \mbox{and}\ \ F_{n}=F_{n-1}+F_{n-2}\ \ \mbox{for}\ \ n\geq2.$$
Indeed, it is noted in \cite{gould} that
$$
Q^n=
\left(\begin{matrix}
1 & 1 \\
1 & 0
\end{matrix}\right)^n
=
\left(\begin{matrix}
F_{n} & F_{n-1} \\
F_{n-1} & F_{n-2}
\end{matrix}\right).
$$
Using this relation above, some familiar identities can be obtained. For instance,
$$
\det\left(\begin{matrix}
F_{n+1} & F_n \\
F_n & F_{n-1}
\end{matrix}\right)
=
\left(\det\left(\begin{matrix}
1 & 1 \\
1 & 0
\end{matrix}\right)\right)^n$$
implies the Cassini identity
$$F_{n+1}F_{n-1}-F_n^2=(-1)^n.$$
Also, using this equality $Q^{n+m}=Q^nQ^m$, one can deduce the Honsberger formula.

In the following, we will apply some similar idea of $Q$ matrix from the numerical cases \cite{lin12} to the Fibonacci type polynomials.
For $n\geq2$ and $x\in\Bbb{R}$, the Fibonacci type polynomial $\mathcal{F}_n(x)$ is defined by
\begin{equation}\label{f}
\mathcal{F}_{0}(x)=0,\ \mathcal{F}_{1}(x)=a\ \mbox{and}\ \mathcal{F}_{n}(x)=c(x)\mathcal{F}_{n-1}(x)+d(x)\mathcal{F}_{n-2}(x)
\end{equation}
where $a\in\Bbb{R}\setminus\{0\}$. Then
$$\left(\begin{matrix}
\mathcal{F}_{n+2}(x) \\
\mathcal{F}_{n+1}(x)
\end{matrix}\right)
=
\left(\begin{matrix}
c(x) & d(x) \\
1 & 0 \\
\end{matrix}\right)
\left(\begin{matrix}
\mathcal{F}_{n+1}(x) \\
\mathcal{F}_{n}(x)
\end{matrix}\right)
.$$
Here we define {\it the Fibonacci type $Q(x)$ matrix} by
$$Q(x)=
\left(\begin{matrix}
c(x) & d(x) \\
1 & 0 \\
\end{matrix}\right).
$$
We note that if $\mathcal{F}_{n}(x)=P_n(x)$ is the Pell polynomial as defined in Table \ref{table}, then
$$Q(x)=
\left(\begin{matrix}
2x & 1 \\
1 & 0 \\
\end{matrix}\right)
$$
which appeared in \cite{hm85}.
In addition, we observe that
$$\left(\begin{matrix}
\mathcal{F}_{n+2}(x) \\
\mathcal{F}_{n+1}(x)
\end{matrix}\right)
=
\left(\begin{matrix}
c(x) & d(x) \\
1 & 0
\end{matrix}\right)^{n}
\left(\begin{matrix}
\mathcal{F}_{2}(x) \\
\mathcal{F}_{1}(x)
\end{matrix}\right)
=
\left(\begin{matrix}
c(x) & d(x) \\
1 & 0
\end{matrix}\right)^{n}
\left(\begin{matrix}
a c(x) \\
a
\end{matrix}\right).$$
On the other hand,
$$\left(\begin{matrix}
\mathcal{F}_{n+2}(x) \\
\mathcal{F}_{n+1}(x)
\end{matrix}\right)
=
\left(\begin{matrix}
c(x)\mathcal{F}_{n+1}(x)+d(x)\mathcal{F}_n(x) \\
c(x)\mathcal{F}_{n}(x)+d(x)\mathcal{F}_{n-1}(x)
\end{matrix}\right)
=
\left(\begin{matrix}
\frac{1}{a}\mathcal{F}_{n+1}(x) & \frac{d(x)}{a}\mathcal{F}_n(x) \\
\frac{1}{a}\mathcal{F}_n(x) & \frac{d(x)}{a}\mathcal{F}_{n-1}(x)
\end{matrix}\right)
\left(\begin{matrix}
a c(x) \\
a
\end{matrix}\right).$$
Hence we have the following result.

\begin{theorem}\label{f-type}
Let $\mathcal{F}_n(x)$ be the Fibonacci type polynomial as defined in Eq. {\rm (\ref{f})}.
Then for each $n\in\Bbb{N}$,
$$
\left(\begin{matrix}
\frac{1}{a}\mathcal{F}_{n+1}(x) & \frac{d(x)}{a}\mathcal{F}_n(x) \\
\frac{1}{a}\mathcal{F}_n(x) & \frac{d(x)}{a}\mathcal{F}_{n-1}(x)
\end{matrix}\right)
=
\left(\begin{matrix}
c(x) & d(x) \\
1 & 0
\end{matrix}\right)^n
=Q^n(x).$$
\end{theorem}
\begin{proof}
Let $n=1$. Then
$$
\left(\begin{matrix}
\frac{1}{a}\mathcal{F}_{2}(x) & \frac{d(x)}{a}\mathcal{F}_1(x) \\
\frac{1}{a}\mathcal{F}_1(x) & \frac{d(x)}{a}\mathcal{F}_{0}(x)
\end{matrix}\right)
=
\left(\begin{matrix}
c(x) & d(x) \\
1 & 0
\end{matrix}\right).$$
Assume the equality holds for $n=k$.
Then we have
$$
\left(\begin{matrix}
\frac{1}{a}\mathcal{F}_{k+1}(x) & \frac{d(x)}{a}\mathcal{F}_k(x) \\
\frac{1}{a}\mathcal{F}_k(x) & \frac{d(x)}{a}\mathcal{F}_{k-1}(x)
\end{matrix}\right)
=
\left(\begin{matrix}
c(x) & d(x) \\
1 & 0
\end{matrix}\right)^k.$$
If $n=k+1$, then
$$
\left(\begin{matrix}
\frac{1}{a}\mathcal{F}_{k+2}(x) & \frac{d(x)}{a}\mathcal{F}_{k+1}(x) \\
\frac{1}{a}\mathcal{F}_{k+1}(x) & \frac{d(x)}{a}\mathcal{F}_{k}(x)
\end{matrix}\right)
=
\left(\begin{matrix}
c(x) & d(x) \\
1 & 0
\end{matrix}\right)
\left(\begin{matrix}
\frac{1}{a}\mathcal{F}_{k+1}(x) & \frac{d(x)}{a}\mathcal{F}_k(x) \\
\frac{1}{a}\mathcal{F}_k(x) & \frac{d(x)}{a}\mathcal{F}_{k-1}(x)
\end{matrix}\right)
=
\left(\begin{matrix}
c(x) & d(x) \\
1 & 0
\end{matrix}\right)^{k+1}.$$
By induction, the result follows.
\end{proof}

The Cassini identity of the Fibonacci type polynomial $\mathcal{F}_n(x)$ can be obtained below by Theorem \ref{f-type}.

\begin{corollary}\label{f-Cassini}
Let $\mathcal{F}_n(x)$ be the Fibonacci type polynomial.
Then for each $n\in\Bbb{N}$,
$$\mathcal{F}_n^2(x)-\mathcal{F}_{n+1}(x)\mathcal{F}_{n-1}(x)=a^2(-d(x))^{n-1}.$$
\end{corollary}
\begin{proof}
By Theorem \ref{f-type}, we have
$$
\det\left(\begin{matrix}
\frac{1}{a}\mathcal{F}_{n+1}(x) & \frac{d(x)}{a}\mathcal{F}_n(x) \\
\frac{1}{a}\mathcal{F}_n(x) & \frac{d(x)}{a}\mathcal{F}_{n-1}(x)
\end{matrix}\right)
=
\left(\det\left(\begin{matrix}
c(x) & d(x) \\
1 & 0
\end{matrix}\right)\right)^n.$$
Hence
$$\mathcal{F}_n^2(x)-\mathcal{F}_{n+1}(x)F_{n-1}(x)=a^2(-d(x))^{n-1}.$$
\end{proof}

\begin{example}
Let $a=1, c(x)=x, d(x)=1$ in Eq. (\ref{f}). Then $\mathcal{F}_n(x)$ is the classical Fibonacci polynomial $F_n(x)$. By Corollary \ref{f-Cassini}, we recover the Cassini identity in \cite{fp09},
$$F_{n+1}(x)F_{n-1}(x)-F_n^2(x)=(-1)^{n}.$$
\end{example}

\begin{example}
Let $\mathcal{F}_n(x)$ be the Pell polynomial $P_n(x)$ as defined in Table \ref{table}. By Corollary \ref{f-Cassini},
$$P_{n+1}(x)P_{n-1}(x)-P_n^2(x)=(-1)^{n}$$
which is the identity (2.5) in \cite{hm85}.
\end{example}

\begin{example}
Let $a=1, c(x)=1, d(x)=2x$ in Eq. (\ref{f}). Then $\mathcal{F}_n(x)=J_n(x)$ is the Jacobsthal polynomial as defined in Table \ref{table}. By Corollary \ref{f-Cassini}, one can obtain the Cassini identity for the Jacobsthal polynomial below
$$J_n^2(x)-J_{n+1}(x)J_{n-1}(x)=(-2x)^{n-1}.$$
\end{example}

By Corollary \ref{f-Cassini}, we have the result below.

\begin{corollary}
Let $\mathcal{F}_n(x)$ be the Fibonacci type polynomial.
Then for each $n\in\Bbb{N}$,
$$\mathcal{F}^2_n(x)-c(x)\mathcal{F}_{n}(x)\mathcal{F}_{n-1}(x)-d(x)\mathcal{F}^2_{n-1}(x)-=a^2(-d(x))^{n-1}.$$
\end{corollary}
\begin{proof}
By
$$\mathcal{F}_n^2(x)-\mathcal{F}_{n+1}(x)\mathcal{F}_{n-1}(x)=a^2(-d(x))^{n-1}.$$
and
$$\mathcal{F}_{n+1}(x)=c(x)\mathcal{F}_{n}(x)+d(x)\mathcal{F}_{n-1}(x),$$
we have
\begin{eqnarray*}
a^2(-d(x))^{n-1}&=&\mathcal{F}_n^2(x)-\left(c(x)\mathcal{F}_{n}(x)+d(x)\mathcal{F}_{n-1}(x)\right)\mathcal{F}_{n-1}(x)\\
&=&\mathcal{F}^2_n(x)-c(x)\mathcal{F}_{n}(x)\mathcal{F}_{n-1}(x)-d(x)\mathcal{F}^2_{n-1}(x).
\end{eqnarray*}
\end{proof}

By applying $Q^{n+m}(x)=Q^n(x)Q^m(x)$, we give the Honsberger's formula for the the Fibonacci type polynomials below.

\begin{corollary}\label{f-Honsberger}
Let $\mathcal{F}_n(x)$ be the Fibonacci type polynomial.
Then for each $n,m\in\Bbb{N}$,
$$a\mathcal{F}_{n+m}(x)=\mathcal{F}_{n}(x)\mathcal{F}_{m+1}(x)+d(x)\mathcal{F}_{n-1}(x)\mathcal{F}_{m}(x).$$
\end{corollary}
\begin{proof}
By
$$\left(\begin{matrix}
c(x) & d(x) \\
1 & 0
\end{matrix}\right)^{n+m}
=\left(\begin{matrix}
c(x) & d(x) \\
1 & 0
\end{matrix}\right)^{n}
\left(\begin{matrix}
c(x) & d(x) \\
1 & 0
\end{matrix}\right)^{m},$$
we have
$$
\left(\begin{matrix}
\frac{1}{a}\mathcal{F}_{n+m+1}(x) & \frac{d(x)}{a}\mathcal{F}_{n+m}(x) \\
\frac{1}{a}\mathcal{F}_{n+m}(x) & \frac{d(x)}{a}\mathcal{F}_{n+m-1}(x)
\end{matrix}\right)
=
\left(\begin{matrix}
\frac{1}{a}\mathcal{F}_{n+1}(x) & \frac{d(x)}{a}\mathcal{F}_n(x) \\
\frac{1}{a}\mathcal{F}_n(x) & \frac{d(x)}{a}\mathcal{F}_{n-1}(x)
\end{matrix}\right)
\left(\begin{matrix}
\frac{1}{a}\mathcal{F}_{m+1}(x) & \frac{d(x)}{a}\mathcal{F}_m(x) \\
\frac{1}{a}\mathcal{F}_m(x) & \frac{d(x)}{a}\mathcal{F}_{m-1}(x)
\end{matrix}\right)
.$$
Hence considering the $(2,1)$ entry of the first matrix in the equality above,
$$a\mathcal{F}_{n+m}(x)=\mathcal{F}_{n}(x)\mathcal{F}_{m+1}(x)+d(x)\mathcal{F}_{n-1}(x)\mathcal{F}_{m}(x).$$
\end{proof}

\begin{remark}
\item (i) Let $a=1$ in Corollary \ref{f-Honsberger}. Then Corollary \ref{f-Honsberger} is the same with the first result of \cite[Proposition 1]{fmm18}, and a generalization of \cite[Proposition 5]{fp09}.
\item (ii) If $m=n-1$ in the above corollary, then for each $n\in\Bbb{N}$,
$$a\mathcal{F}_{2n-1}(x)=\mathcal{F}^2_{n}(x)+d(x)\mathcal{F}^2_{n-1}(x)$$
which generalizes the numerical case of Fibonacci sequences.
\end{remark}

\begin{example}
Let $a=1, c(x)=x, d(x)=1$ in Eq. (\ref{f}). Then $\mathcal{F}_n(x)=F_n(x)$ is the Fibonacci polynomial as defined in Table \ref{table}. By Corollary \ref{f-Honsberger}, we recover the Honsberger formula in \cite[Proposition 5]{fp09},
$${F}_{n+m}(x)={F}_{n}(x){F}_{m+1}(x)+{F}_{n-1}(x){F}_{m}(x).$$
\end{example}

\begin{example}
Let $a=1, c(x)=2x, d(x)=1$ in Eq. (\ref{f}). Then $\mathcal{F}_n(x)$ is the Pell polynomial $P_n(x)$. By Corollary \ref{f-Honsberger}, we have
$${P}_{n+m}(x)={P}_{n}(x){P}_{m+1}(x)+{P}_{n-1}(x){P}_{m}(x)$$
which is the equality (3.14) in \cite{hm85}.
\end{example}

Using $Q^{n-m}(x)=Q^n(x)Q^{-m}(x)$ for $n\geq m$, we next will prove the d'Ocagne identity for $\mathcal{F}_n(x)$. Here we need to assume $d(x)\neq0$ for each $x\in \Bbb{R}$ so that $Q(x)$ is invertible. Moreover, note that
$$Q^{-m}(x)
=
\left(\begin{matrix}
\frac{1}{a}\mathcal{F}_{m+1}(x) & \frac{d(x)}{a}\mathcal{F}_m(x) \\
\frac{1}{a}\mathcal{F}_m(x) & \frac{d(x)}{a}\mathcal{F}_{m-1}(x)
\end{matrix}\right)^{-1}
=\frac{1}{(-d(x))^m}
\left(\begin{matrix}
\frac{d(x)}{a}\mathcal{F}_{m-1}(x) & -\frac{d(x)}{a}\mathcal{F}_m(x) \\
-\frac{1}{a}\mathcal{F}_m(x) & \frac{1}{a}\mathcal{F}_{m+1}(x)
\end{matrix}\right)
$$
by Theorem \ref{f-type} and Corollary \ref{f-Cassini}.

\begin{corollary}\label{f-d'Ocagne}
Let $\mathcal{F}_n(x)$ be the Fibonacci type polynomial, and let $d(x)\neq0$ for each $x\in \Bbb{R}$.
Then for each $n,m\in\Bbb{N}$ with $n\geq m$,
$$a(-d(x))^m\mathcal{F}_{n-m}(x)=\mathcal{F}_{n}(x)\mathcal{F}_{m+1}(x)-\mathcal{F}_{n+1}(x)\mathcal{F}_{m}(x).$$
\end{corollary}
\begin{proof}
By
$Q^{n-m}(x)=Q^n(x)Q^{-m}(x)$,
we have
\begin{eqnarray*}
&&\left(\begin{matrix}
\frac{1}{a}\mathcal{F}_{n-m+1}(x) & \frac{d(x)}{a}\mathcal{F}_{n-m}(x) \\
\frac{1}{a}\mathcal{F}_{n-m}(x) & \frac{d(x)}{a}\mathcal{F}_{n-m-1}(x)
\end{matrix}\right)\\
&=&
\left(\begin{matrix}
\frac{1}{a}\mathcal{F}_{n+1}(x) & \frac{d(x)}{a}\mathcal{F}_n(x) \\
\frac{1}{a}\mathcal{F}_n(x) & \frac{d(x)}{a}\mathcal{F}_{n-1}(x)
\end{matrix}\right)
\frac{1}{(-d(x))^m}
\left(\begin{matrix}
\frac{d(x)}{a}\mathcal{F}_{m-1}(x) & -\frac{d(x)}{a}\mathcal{F}_m(x) \\
-\frac{1}{a}\mathcal{F}_m(x) & \frac{1}{a}\mathcal{F}_{m+1}(x)
\end{matrix}\right)
.\end{eqnarray*}
Hence considering the $(1,2)$ entry of the first matrix in the equality above,
$$a(-d(x))^m\mathcal{F}_{n-m}(x)=\mathcal{F}_{n}(x)\mathcal{F}_{m+1}(x)-\mathcal{F}_{n+1}(x)\mathcal{F}_{m}(x).$$
\end{proof}

\begin{example}
Let $\mathcal{F}_n(x)$ be the Fibonacci polynomial $F_n(x)$ as defined in Table \ref{table}. By Corollary \ref{f-d'Ocagne},
$$(-1)^m{F}_{n-m}(x)={F}_{n}(x){F}_{m+1}(x)-{F}_{n+1}(x){F}_{m}(x)$$
which is the d'Ocagne identity in \cite[Corollary 8]{fp09}, and the identity (47) of \cite[Proposition 3]{fmm18}.
\end{example}

We note that $Q(x)=
\left(\begin{matrix}
c(x) & d(x) \\
1 & 0
\end{matrix}\right)$
satisfies
$Q^2(x)=c(x)Q(x)+d(x)I$
where
$I=
\left(\begin{matrix}
1 & 0 \\
0 & 1
\end{matrix}\right)$.
Using this equality, one can obtain the following expression of $\mathcal{F}_n(x)$.

\begin{theorem}\label{f-series}
Let $\mathcal{F}_n(x)$ be the Fibonacci type polynomial.
Then for each $n,p\in\Bbb{N}$,
$$\mathcal{F}_{2n+p}(x)
=\sum_{j=0}^{n}
\left(\begin{matrix}
n \\
j
\end{matrix}\right)
c^j(x)d^{n-j}(x)\mathcal{F}_{j+p}(x).$$
\end{theorem}
\begin{proof}
Consider
\begin{eqnarray*}
&&\left(\begin{matrix}
\frac{1}{a}\mathcal{F}_{2n+p+1}(x) & \frac{d(x)}{a}\mathcal{F}_{2n+p}(x) \\
\frac{1}{a}\mathcal{F}_{2n+p}(x) & \frac{d(x)}{a}\mathcal{F}_{2n+p-1}(x)
\end{matrix}\right)\\
&=&Q^{2n+p}(x)\\
&=&Q^{p}(x)\left(Q^{2}(x)\right)^n\\
&=&Q^{p}(x)\left(c(x)Q(x)+d(x)I\right)^n\\
&=&Q^{p}(x)\left(\sum_{j=0}^n
\left(\begin{matrix}
n \\
j
\end{matrix}\right)
c^j(x)d^{n-j}(x)Q^j(x)\right)\\
&=&
\left(\begin{matrix}
\frac{1}{a}\mathcal{F}_{p+1}(x) & \frac{d(x)}{a}\mathcal{F}_{p}(x) \\
\frac{1}{a}\mathcal{F}_{p}(x) & \frac{d(x)}{a}\mathcal{F}_{p-1}(x)
\end{matrix}\right)
\cdot
\sum_{j=0}^n
\left(\begin{matrix}
n \\
j
\end{matrix}\right)
c^j(x)d^{n-j}(x)
\left(\begin{matrix}
\frac{1}{a}\mathcal{F}_{j+1}(x) & \frac{d(x)}{a}\mathcal{F}_{j}(x) \\
\frac{1}{a}\mathcal{F}_{j}(x) & \frac{d(x)}{a}\mathcal{F}_{j-1}(x)
\end{matrix}\right).
\end{eqnarray*}
Then by Corollary \ref{f-Honsberger} and the $(1,2)$ entry of the first matrix in the above equality, we have
\begin{eqnarray*}
a\mathcal{F}_{2n+p}(x)&=&\sum_{j=0}^n
\left(\begin{matrix}
n \\
j
\end{matrix}\right)
c^j(x)d^{n-j}(x)\left(\mathcal{F}_{p}(x)\mathcal{F}_{j+1}(x)+d(x)\mathcal{F}_{p-1}(x)\mathcal{F}_{j}(x)\right)\\
&=&a\sum_{j=0}^{n}
\left(\begin{matrix}
n \\
j
\end{matrix}\right)
c^j(x)d^{n-j}(x)\mathcal{F}_{j+p}(x).
\end{eqnarray*}
\end{proof}

\begin{example}
Let $\mathcal{F}_n(x)$ be the Fibonacci polynomial $F_n(x)$ in which $a=1, c(x)=x, d(x)=1$ in Eq. (\ref{f}). By Theorem \ref{f-series}, we have
$${F}_{2n+p}(x)
=\sum_{j=0}^{n}
\left(\begin{matrix}
n \\
j
\end{matrix}\right)
x^j{F}_{j+p}(x).$$
Given $n=2$ and $p=1$, we have $$F_5(x)=F_1(x)+2xF_2(x)+x^2F_3(x).$$
Indeed, this equality holds for $F_1(x)=1, F_2(x)=x, F_3(x)=x^2+1$ and $F_5(x)=x^4+3x^2+1$.
\end{example}

\section{Lucas type polynomials}

Based on the results of Fibonacci type polynomials, some identities of Lucas type polynomials will be demonstrated in this section.
Throughout this section, we assume $\mathcal{L}_{n}(x)$ and $\mathcal{F}_{n}(x)$ have the same recursive formula with $\mathcal{L}_0(x)=\mathcal{F}_1(x)$, that is,
for $n\geq2$,
$$\mathcal{F}_{0}(x)=0,\ \mathcal{F}_{1}(x)=a\ \mbox{and}\ \mathcal{F}_{n}(x)=c(x)\mathcal{F}_{n-1}(x)+d(x)\mathcal{F}_{n-2}(x),$$
and
\begin{equation}\label{l}
\mathcal{L}_{0}(x)=a,\ \mathcal{L}_{1}(x)=b(x)\ \mbox{and}\ \mathcal{L}_{n}(x)=c(x)\mathcal{L}_{n-1}(x)+d(x)\mathcal{L}_{n-2}(x)
\end{equation}
where $a\in\Bbb{R}\setminus\{0\}$. By applying Theorem \ref{f-type}, one can connect $\mathcal{L}_{n}(x)$ with $\mathcal{F}_{n}(x)$ below.

\begin{theorem}\label{lf}
Let $\mathcal{F}_n(x)$ and $\mathcal{L}_n(x)$ be the Fibonacci type polynomial and Lucas type polynomial respectively with $\mathcal{L}_0(x)=\mathcal{F}_1(x)=a$.
Then for each $n\in\Bbb{N}$,
$$
\left(\begin{matrix}
\mathcal{L}_{n+2}(x) & \mathcal{L}_{n+1}(x) \\
\mathcal{L}_{n+1}(x) & \mathcal{L}_{n}(x)
\end{matrix}\right)
=
\left(\begin{matrix}
\frac{b(x)c(x)+ad(x)}{a} & \frac{b(x)d(x)}{a} \\
\frac{b(x)}{a}& d
\end{matrix}\right)
\left(\begin{matrix}
\mathcal{F}_{n+1}(x) & \mathcal{F}_n(x) \\
\mathcal{F}_n(x) & \mathcal{F}_{n-1}(x)
\end{matrix}\right)$$
\end{theorem}
\begin{proof}
First, we will prove $\mathcal{L}_n(x)=\frac{b(x)}{a}\mathcal{F}_n(x)+d(x)\mathcal{F}_{n-1}(x)$ holds for each $n\in\Bbb{N}$.
Let $n=1$. Then
$$\mathcal{L}_1(x)=b(x)=\frac{b(x)}{a}\mathcal{F}_1(x)+d(x)\mathcal{F}_{0}(x).$$
Let $n=2$. Then
$$\mathcal{L}_2(x)=b(x)c(x)+ad(x)=\frac{b(x)}{a}\mathcal{F}_2(x)+d(x)\mathcal{F}_{1}(x).$$
Assume this equality hods for $n=k-1$ and $k$.
Let $n=k+1$. Then
\begin{eqnarray*}
\mathcal{L}_{k+1}(x)&=&c(x)\mathcal{L}_{k}(x)+d(x)\mathcal{L}_{k-1}(x)\\
&=&c(x)\left(\frac{b(x)}{a}\mathcal{F}_k(x)+d(x)\mathcal{F}_{k-1}(x)\right)+d(x)\left(\frac{b(x)}{a}\mathcal{F}_{k-1}(x)+d(x)\mathcal{F}_{k-2}(x)\right)\\
&=&\frac{b(x)}{a}\left(c(x)\mathcal{F}_k(x)+d(x)\mathcal{F}_{k-1}(x)\right)+d(x)\left(c(x)\mathcal{F}_{k-1}(x)+d(x)\mathcal{F}_{k-2}(x)\right)\\
&=&\frac{b(x)}{a}\mathcal{F}_{k+1}(x)+d(x)\mathcal{F}_k(x).
\end{eqnarray*}
By induction, $\mathcal{L}_n(x)=\frac{b(x)}{a}\mathcal{F}_n(x)+d(x)\mathcal{F}_{n-1}(x)$ holds for all $n\in\Bbb{N}$.
On the other hand, we have
\begin{eqnarray*}
\mathcal{L}_n(x)&=&\frac{b(x)}{a}\mathcal{F}_n(x)+d(x)\mathcal{F}_{n-1}(x)\\
&=&\frac{b(x)}{a}\left(c(x)\mathcal{F}_{n-1}(x)+d(x)\mathcal{F}_{n-2}(x)\right)+d(x)\mathcal{F}_{n-1}(x)\\
&=& \frac{b(x)c(x)+ad(x)}{a}\mathcal{F}_{n-1}(x)+ \frac{b(x)d(x)}{a}\mathcal{F}_{n-2}(x).
\end{eqnarray*}
One has the result by these two equalities
$$\mathcal{L}_n(x)=\frac{b(x)}{a}\mathcal{F}_n(x)+d(x)\mathcal{F}_{n-1}(x)$$
and
$$\mathcal{L}_n(x)=\frac{b(x)c(x)+ad(x)}{a}\mathcal{F}_{n-1}(x)+ \frac{b(x)d(x)}{a}\mathcal{F}_{n-2}(x).$$
\end{proof}

Next, we will demonstrate the relation between Lucas type polynomials and the Fibonacci type $Q(x)$ matrix .

\begin{theorem}\label{l-type}
Let $\mathcal{L}_n(x)$ be the Lucas type polynomial.
Then for each $n\in\Bbb{N}$,
$$
\left(\begin{matrix}
\mathcal{L}_{n+2}(x) & d(x)\mathcal{L}_{n+1}(x) \\
\mathcal{L}_{n+1}(x) & d(x)\mathcal{L}_{n}(x)
\end{matrix}\right)
=
\left(\begin{matrix}
\mathcal{L}_2(x) & d(x)\mathcal{L}_1(x) \\
\mathcal{L}_1(x) & d(x)\mathcal{L}_0(x)
\end{matrix}\right)
Q^n(x)$$
\end{theorem}
\begin{proof}
By Theorem \ref{f-type} and Theorem \ref{lf}, we have
\begin{eqnarray*}
&&\left(\begin{matrix}
\mathcal{L}_{n+2}(x) & d(x)\mathcal{L}_{n+1}(x) \\
\mathcal{L}_{n+1}(x) & d(x)\mathcal{L}_{n}(x)
\end{matrix}\right)\\
&=&
\left(\begin{matrix}
\mathcal{L}_{n+2}(x) & \mathcal{L}_{n+1}(x) \\
\mathcal{L}_{n+1}(x) & \mathcal{L}_{n}(x)
\end{matrix}\right)
\left(\begin{matrix}
1 & 0 \\
0 & d(x)
\end{matrix}\right)
\\
&=&
\left(\begin{matrix}
\frac{b(x)c(x)+ad(x)}{a} & \frac{b(x)d(x)}{a} \\
\frac{b(x)}{a}& d
\end{matrix}\right)
\left(\begin{matrix}
\mathcal{F}_{n+1}(x) & \mathcal{F}_n(x) \\
\mathcal{F}_n(x) & \mathcal{F}_{n-1}(x)
\end{matrix}\right)
\left(\begin{matrix}
1 & 0 \\
0 & d(x)
\end{matrix}\right)\\
&=&
\left(\begin{matrix}
b(x)c(x)+ad(x) & b(x)d(x) \\
b(x)& ad(x)
\end{matrix}\right)
\left(\begin{matrix}
\frac{1}{a}\mathcal{F}_{n+1}(x) & \frac{d(x)}{a}\mathcal{F}_n(x) \\
\frac{1}{a}\mathcal{F}_n(x) & \frac{d(x)}{a}\mathcal{F}_{n-1}(x)
\end{matrix}\right)\\
&=&
\left(\begin{matrix}
\mathcal{L}_2(x) & d(x)\mathcal{L}_1(x) \\
\mathcal{L}_1(x) & d(x)\mathcal{L}_0(x)
\end{matrix}\right)
\left(\begin{matrix}
c(x) & d(x) \\
1 & 0
\end{matrix}\right)^n
\end{eqnarray*}
for each $n\in\Bbb{N}$.
\end{proof}

Using Theorem \ref{l-type}, one has the Cassini identity for the Lucas type polynomial $\mathcal{L}_n(x)$.

\begin{corollary}\label{l-Cassini}
Let $\mathcal{L}_n(x)$ be the Lucas type polynomial.
Then for each $n\in\Bbb{N}$,
$$\mathcal{L}_{n+2}(x)\mathcal{L}_{n}(x)-\mathcal{L}^2_{n+1}(x)=\left(\mathcal{L}_{2}(x)\mathcal{L}_{0}(x)-\mathcal{L}^2_{1}(x)\right)(-d(x))^n.$$
\end{corollary}
\begin{proof}
By Theorem \ref{l-type}, we have
$$
\det
\left(\begin{matrix}
\mathcal{L}_{n+2}(x) & d(x)\mathcal{L}_{n+1}(x) \\
\mathcal{L}_{n+1}(x) & d(x)\mathcal{L}_{n}(x)
\end{matrix}\right)
=
\det
\left(\begin{matrix}
\mathcal{L}_2(x) & d(x)\mathcal{L}_1(x) \\
\mathcal{L}_1(x) & d(x)\mathcal{L}_0(x)
\end{matrix}\right)
\left(\det\left(\begin{matrix}
c(x) & d(x) \\
1 & 0
\end{matrix}\right)\right)^n.$$
Hence
$$\mathcal{L}_{n+2}(x)\mathcal{L}_{n}(x)-\mathcal{L}^2_{n+1}(x)=\left(\mathcal{L}_{2}(x)\mathcal{L}_{0}(x)-\mathcal{L}^2_{1}(x)\right)(-d(x))^n.$$
\end{proof}

\begin{example}
Let $a=2, b(x)=2x, c(x)=2x, d(x)=1$ in Eq. (\ref{l}). Then $\mathcal{L}_n(x)=D_n(x)$ is the Pell-Lucas polynomial as defined in Table \ref{table}. By Corollary \ref{l-Cassini}, the Cassini identity for the Pell-Lucas polynomial $D_n(x)$ is given by
$$D_{n+2}(x)D_{n}(x)-D^2_{n+1}(x)=(4x^2+4)(-1)^n.$$
\end{example}

By Corollary \ref{l-Cassini}, we have the result below.

\begin{corollary}
Let $\mathcal{L}_n(x)$ be the Lucas type polynomial.
Then for each $n\in\Bbb{N}$,
$$c(x)\mathcal{L}_{n+1}(x)\mathcal{L}_{n}(x)+d(x)\mathcal{L}^2_{n}(x)-\mathcal{L}^2_{n+1}(x)=\left(\mathcal{L}_{2}(x)\mathcal{L}_{0}(x)-\mathcal{L}^2_{1}(x)\right)(-d(x))^n.$$
\end{corollary}
\begin{proof}
By
$$\mathcal{L}_{n+2}(x)\mathcal{L}_{n}(x)-\mathcal{L}^2_{n+1}(x)=\left(\mathcal{L}_{2}(x)\mathcal{L}_{0}(x)-\mathcal{L}^2_{1}(x)\right)(-d(x))^n$$
and
$$\mathcal{L}_{n+2}(x)=c(x)\mathcal{L}_{n+1}(x)+d(x)\mathcal{L}_{n}(x),$$
we have
\begin{eqnarray*}
&&\left(\mathcal{L}_{2}(x)\mathcal{L}_{0}(x)-\mathcal{L}^2_{1}(x)\right)(-d(x))^n\\
&=&\left(c(x)\mathcal{L}_{n+1}(x)+d(x)\mathcal{L}_{n}(x)\right)\mathcal{L}_{n}(x)-\mathcal{L}^2_{n+1}(x)\\
&=&c(x)\mathcal{L}_{n+1}(x)\mathcal{L}_{n}(x)+d(x)\mathcal{L}^2_{n}(x)-\mathcal{L}^2_{n+1}(x).
\end{eqnarray*}
\end{proof}

Using $Q^2(x)=c(x)Q(x)+d(x)I$ again, we have the expression of $\mathcal{L}_{n}(x)$.

\begin{theorem}\label{l-series}
Let $\mathcal{L}_n(x)$ be the Lucas type polynomial.
Then for each $n,p\in\Bbb{N}$,
$$
\mathcal{L}_{2n+p}(x)=\sum_{j=0}^{n}
\left(\begin{matrix}
n \\
j
\end{matrix}\right)
c^j(x)d^{n-j}(x)\mathcal{L}_{p+j}(x).
$$
\end{theorem}
\begin{proof}
By Theorem \ref{l-type}, we have
\begin{eqnarray*}
&&\left(\begin{matrix}
\mathcal{L}_{2n+p+2}(x) & d(x)\mathcal{L}_{2n+p+1}(x) \\
\mathcal{L}_{2n+p+1}(x) & d(x)\mathcal{L}_{2n+p}(x)
\end{matrix}\right)
\\
&=&
\left(\begin{matrix}
\mathcal{L}_2(x) & d(x)\mathcal{L}_1(x) \\
\mathcal{L}_1(x) & d(x)\mathcal{L}_0(x)
\end{matrix}\right)
Q^{2n+p}(x)\\
&=&
\left(\begin{matrix}
\mathcal{L}_2(x) & d(x)\mathcal{L}_1(x) \\
\mathcal{L}_1(x) & d(x)\mathcal{L}_0(x)
\end{matrix}\right)
Q^{p}(x)\left(Q^{2}(x)\right)^n\\
&=&
\left(\begin{matrix}
\mathcal{L}_{p+2}(x) & d(x)\mathcal{L}_{p+1}(x) \\
\mathcal{L}_{p+1}(x) & d(x)\mathcal{L}_{p}(x)
\end{matrix}\right)
\left(c(x)Q(x)+d(x)I\right)^n\\
&=&
\left(\begin{matrix}
\mathcal{L}_{p+2}(x) & d(x)\mathcal{L}_{p+1}(x) \\
\mathcal{L}_{p+1}(x) & d(x)\mathcal{L}_{p}(x)
\end{matrix}\right)
\left(\sum_{j=0}^n
\left(\begin{matrix}
n \\
j
\end{matrix}\right)
c^j(x)d^{n-j}(x)Q^j(x)\right)\\
&=&
\sum_{j=0}^n
\left(\begin{matrix}
n \\
j
\end{matrix}\right)
c^j(x)d^{n-j}(x)
\left(\begin{matrix}
\mathcal{L}_{p+j+2}(x) & d(x)\mathcal{L}_{p+j+1}(x) \\
\mathcal{L}_{p+j+1}(x) & d(x)\mathcal{L}_{p+j}(x)
\end{matrix}\right)
\end{eqnarray*}
By considering the $(2,2)$ entry of the first matrix in the above equality, we have
$$
\mathcal{L}_{2n+p}(x)=\sum_{j=0}^{n}
\left(\begin{matrix}
n \\
j
\end{matrix}\right)
c^j(x)d^{n-j}(x)\mathcal{L}_{p+j}(x).
$$
\end{proof}

\begin{example}
Let $\mathcal{L}_n(x)$ be the Morgan-Voyce polynomial $C_n(x)$ in which $a=2, b(x)=x+2, c(x)=x+2, d(x)=-1$ in Eq. (\ref{l}). By Theorem \ref{l-series}, we have
$${C}_{2n+p}(x)
=\sum_{j=0}^{n}
\left(\begin{matrix}
n \\
j
\end{matrix}\right)
(x+2)^j(-1)^{n-j}{C}_{p+j}(x).$$
\end{example}

Finally, we end up this note by providing an identity in which $\mathcal{F}_n(x)$ and $\mathcal{L}_n(x)$ are involved.

\begin{proposition}\label{l-Honsberger}
Let $\mathcal{F}_n(x)$ and $\mathcal{L}_n(x)$ be the Fibonacci type polynomial and Lucas type polynomial respectively with $\mathcal{L}_0(x)=\mathcal{F}_1(x)=a$.
Then for each $n,m\in\Bbb{N}$,
$$a\mathcal{L}_{n+m}(x)=\mathcal{L}_{n+1}(x)\mathcal{F}_{m}(x)+d(x)\mathcal{L}_{n}(x)\mathcal{F}_{m-1}(x).$$
\end{proposition}
\begin{proof}
By Theorem \ref{l-type}, we have
\begin{eqnarray*}
&&\left(\begin{matrix}
\mathcal{L}_{n+m+2}(x) & d(x)\mathcal{L}_{n+m+1}(x) \\
\mathcal{L}_{n+m+1}(x) & d(x)\mathcal{L}_{n+m}(x)
\end{matrix}\right)\\
&=&
\left(\begin{matrix}
\mathcal{L}_2(x) & d(x)\mathcal{L}_1(x) \\
\mathcal{L}_1(x) & d(x)\mathcal{L}_0(x)
\end{matrix}\right)
Q^{n}(x)Q^{m}(x)\\
&=&
\left(\begin{matrix}
\mathcal{L}_{n+2}(x) & d(x)\mathcal{L}_{n+1}(x) \\
\mathcal{L}_{n+1}(x) & d(x)\mathcal{L}_{n}(x)
\end{matrix}\right)
\left(\begin{matrix}
\frac{1}{a}\mathcal{F}_{m+1}(x) & \frac{d(x)}{a}\mathcal{F}_{m}(x) \\
\frac{1}{a}\mathcal{F}_{m}(x) & \frac{d(x)}{a}\mathcal{F}_{m-1}(x)
\end{matrix}\right).
\end{eqnarray*}
Then by the $(2,2)$ entry of the first matrix in the above equality, we have
$$a\mathcal{L}_{n+m}(x)=\mathcal{L}_{n+1}(x)\mathcal{F}_{m}(x)+d(x)\mathcal{L}_{n}(x)\mathcal{F}_{m-1}(x)$$
for each $n,m\in\Bbb{N}$.
\end{proof}

\begin{example}
Let $\mathcal{F}_n(x)$ and $\mathcal{L}_n(x)$ be the Jacobsthal polynomial $J_n(x)$ and the Jacobsthal-Lucas polynomial $\Lambda_n(x)$ respectively, as defined in Table \ref{table}. Then $\Lambda_0(x)=J_1(x)=1$ which satisfies the condition in Proposition \ref{l-Honsberger}. Hence we have the following equality for $J_n(x)$ and $\Lambda_n(x)$:
$$\Lambda_{n+m}(x)=\Lambda_{n+1}(x)J_{m}(x)+2x\Lambda_{n}(x)J_{m-1}(x).$$
\end{example}

\begin{proposition}\label{l-d'Ocagne}
Let $\mathcal{F}_n(x)$ and $\mathcal{L}_n(x)$ be the Fibonacci type polynomial and Lucas type polynomial respectively with $\mathcal{L}_0(x)=\mathcal{F}_1(x)=a$.
Let $d(x)\neq0$ for each $x\in\Bbb{R}$.
Then for each $n,m\in\Bbb{N}$ with $n\geq m$,
$$a(-d(x))^{m}\mathcal{L}_{n-m}(x)=\mathcal{L}_{n}(x)\mathcal{F}_{m+1}(x)-\mathcal{L}_{n+1}(x)\mathcal{F}_{m}(x).$$
\end{proposition}
\begin{proof}
By Theorem \ref{l-type} and $Q^{n-m}(x)=Q^{n}(x)Q^{-m}(x)$, we have
\begin{eqnarray*}
&&\left(\begin{matrix}
\mathcal{L}_{n-m+2}(x) & d(x)\mathcal{L}_{n-m+1}(x) \\
\mathcal{L}_{n-m+1}(x) & d(x)\mathcal{L}_{n-m}(x)
\end{matrix}\right)\\
&=&
\left(\begin{matrix}
\mathcal{L}_2(x) & d(x)\mathcal{L}_1(x) \\
\mathcal{L}_1(x) & d(x)\mathcal{L}_0(x)
\end{matrix}\right)
Q^{n}(x)Q^{-m}(x)\\
&=&
\left(\begin{matrix}
\mathcal{L}_{n+2}(x) & d(x)\mathcal{L}_{n+1}(x) \\
\mathcal{L}_{n+1}(x) & d(x)\mathcal{L}_{n}(x)
\end{matrix}\right)
\frac{1}{(-d(x))^m}
\left(\begin{matrix}
\frac{d(x)}{a}\mathcal{F}_{m-1}(x) & -\frac{d(x)}{a}\mathcal{F}_{m}(x) \\
-\frac{1}{a}\mathcal{F}_{m}(x) & \frac{1}{a}\mathcal{F}_{m+1}(x)
\end{matrix}\right).
\end{eqnarray*}
Then considering the $(2,2)$ entry of the first matrix in the above equality, we have
$$a(-d(x))^{m}\mathcal{L}_{n-m}(x)=\mathcal{L}_{n}(x)\mathcal{F}_{m+1}(x)-\mathcal{L}_{n+1}(x)\mathcal{F}_{m}(x).$$
\end{proof}

\begin{example}
Let $\mathcal{F}_n(x)$ and $\mathcal{L}_n(x)$ be the Jacobsthal polynomial $J_n(x)$ and the Jacobsthal-Lucas polynomial $\Lambda_n(x)$ respectively.
Then $\Lambda_0(x)=J_1(x)=1$ and
$$(-2x)^{m}\Lambda_{n-m}(x)=\Lambda_{n}(x)J_{m+1}(x)-\Lambda_{n+1}(x)J_{m}(x).$$
\end{example}

\vspace{.1in}
\end{document}